\documentclass[11pt,reqno]{amsart}
\usepackage{amssymb,latexsym}
\usepackage{url,units,color}
\usepackage[margin=1in]{geometry}
\newcommand{\EM}{Erd\H{o}s-Moser}
\newcommand{\EME}{\EM\ equation}

\DeclareMathOperator{\lcm}{lcm}
\newcommand{\emdash}{\hspace{2pt}\textemdash\hspace{2pt}}

\newcommand{\K}{K\hspace{-0.1em}}
\newcommand{\subPK}{p\hspace{0.1em}\mid\hspace{0.05em}K}
\newcommand{\subPk}{p\hspace{0.1em}\mid\hspace{0.05em}k}

\newcommand{\subPL}{p\hspace{0.1em}\mid\hspace{0.05em}L}

\newcommand{\ie}{\textit{i.e.,\ }}
\newcommand{\ppp}{primary pseudoperfect number}
\newcommand{\PPN}{PPN}
\newcommand{\AP}{AP}

\theoremstyle{plain}
\newtheorem{thm}{Theorem}

\newtheorem{prop}{Proposition}
\theoremstyle{definition}
\newtheorem{dfn}{Definition}

\newtheorem{cnj}{Conjecture}
\newtheorem*{notation*}{Notation}
\newtheorem*{remark*}{Remark}

\begin{document}
\title{Primary Pseudoperfect Numbers,\\Arithmetic Progressions, \\and the Erd\H{o}s-Moser Equation}
\markright{Primary Pseudoperfect Numbers and Progressions}
\author{Jonathan Sondow and Kieren MacMillan}

\begin{abstract}
A \emph{\ppp} (\PPN) is an integer $K > 1$ satisfying the equation
\begin{equation*}
	\frac{1}{K} + \sum_{\subPK} \frac{1}{p} = 1,
\end{equation*}
where $p$ denotes a prime.
PPNs arise in studying perfectly weighted graphs  and singularities of algebraic surfaces, and are related to Sylvester's sequence, Giuga numbers, Zn\'am's problem, the inheritance problem, and Curtiss's bound on solutions of a unit fraction equation.

Here we show $K \equiv 6 \! \pmod{6^2}$ if $6\mid K$, and uncover a remarkable $7$-term arithmetic progression of residues modulo $6^2\cdot8$ in the sequence of known PPNs. On that basis, we pose a conjecture which leads to a conditional proof of the new record lower bound $k>10^{3.99\times10^{20}}$ on any non-trivial solution to the Erd\H{o}s-Moser Diophantine equation $1^n + 2^n + \dotsb + k^n = (k+1)^n$.
\end{abstract}

\maketitle

\section{{\bf INTRODUCTION.}}  \label{SEC: INTRO}
In 1922 Curtiss~\cite{Curtiss} proved Kellogg's \cite{Kellogg} conjectured bound on solutions to a \emph{unit fraction} equation
\begin{equation}
	\sum_{i=1}^n\frac{1}{x_i} = 1 \quad \implies \quad \max_{1\, \le\, i\,\le n} x_i \le S_n - 1 \label{EQ: Curtiss bound}
\end{equation}
where \emph{Sylvester's sequence} \cite{Anne, Sylvester, wiki}, \cite[A000058]{Sloane},
\begin{equation}  \label{EQ: SS}
\begin{aligned}
	S_n = 2, 3, 7, 43, 1807, 3263443, 10650056950807,
		  113423713055421844361000443,\dotsc,
\end{aligned}
\end{equation}
is defined by the recurrence $S_n = S_1S_2\dotsb S_{n-1}+1$, with $S_1=2$.

The equation in \eqref{EQ: Curtiss bound} also appears in finite group theory. Suppose we have a finite group $G$, and assume it has conjugacy classes $C_1,\dotsc,C_n$. The number of elements of $C_i$ divides the order $N$ of $G,$ so we can write $\#C_i = N/m_i$ with $m_i$ an integer and
\begin{equation*}
	N = \#C_1 + \dotsb + \#C_n = \frac{N}{m_1} + \dotsb + \frac{N}{m_n}.
\end{equation*}
It follows that $1 = \sum_i 1/m_i$. Curtiss's result now says that the number of groups with a prescribed number $n$ of conjugacy classes is finite. For more on this, see Landau \cite{Landau} or Lenstra \cite{Lenstra}.

The present article is concerned with the particular unit fraction equation
\begin{equation}
	\frac{1}{K} + \sum_{\subPK} \frac{1}{p} = 1.  \label{EQ: PPP def}
\end{equation}
Here and throughout the paper, $p$ denotes a prime. Equation~\eqref{EQ: PPP def} is related to perfectly weighted graphs \cite{Butske} and singularities of algebraic surfaces \cite{BH}. The companion equation
\begin{equation*}
	\frac{-1}{L} + \sum_{\subPL} \frac{1}{p} = 1
\end{equation*}
occurs in the study of \emph{Giuga numbers} \cite{BBBG,ST}, \cite[A17]{Guy}, \cite[A007850]{Sloane}, and a generalization of \eqref{EQ: PPP def},
\begin{equation*}
	\prod_{i=1}^r \frac{1}{x_i} + \sum_{i=1}^r \frac{1}{x_i} = 1,
\end{equation*}
arises in \emph{Zn\'am's problem} \cite{BV, Cao}, \cite[A075461]{Sloane} and \emph{the inheritance problem} \cite{Anne}. See also \cite{ACF} for recent work on the equation in \eqref{EQ: Curtiss bound}.

In Section~2, we summarize the known facts about solutions to the unit fraction equation~\eqref{EQ: PPP def}. In Section~\ref{SEC: APS}, we reduce the solutions modulo $288$ and uncover a remarkable $7$-term arithmetic progression of residues, leading to two conjectures. In the final section, we relate solutions of~\eqref{EQ: PPP def} to possible solutions of the \emph{Erd\H{o}s-Moser Diophantine equation}
\begin{equation}
	1^n + 2^n + \dotsb + (k-1)^n+k^n = (k+1)^n.  \label{EQ: EME}
\end{equation}
Assuming a weak form of one of our conjectures, we give a conditional proof of a new record lower bound on any non-trivial solution of \eqref{EQ: EME}.

\section{{\bf PRIMARY PSEUDOPERFECT NUMBERS.}}  \label{SEC: PPNS}

Recall that a positive integer is called \emph{perfect} if it is the sum of \emph{all} of its proper divisors, and \emph{pseudoperfect} if it is the sum of \emph{some} of its proper divisors~\cite[B1, B2]{Guy}, \cite[A000396, A005835]{Sloane}.

\begin{dfn}[{\bf Butske, Jaje, and Mayernik \cite{Butske}}]  \label{DFN: ppp}
A \emph{\ppp} (\PPN{} for short) is an integer $K > 1$ that satisfies the unit fraction equation~\eqref{EQ: PPP def}. See \cite{pm, Weisstein, wiki} and \cite[A054377]{Sloane}. Note that, just as $1$ is not a prime number, so too $1$ is not a PPN.
\end{dfn}

Multiplying equation \eqref{EQ: PPP def} by $\K$ gives the equivalent integer condition
\begin{equation}
	1 + \sum_{\subPK} \frac{K}{p} = K. \label{EQ: integer def}
\end{equation}
For example, $42=2\cdot3\cdot7$ is a \PPN{}, because  $42/2=21,\ 42/3=14,\ 42/7=6,$ and $1 + 21+14+6=42$. From \eqref{EQ: integer def}, we see that {\em all \PPN{s} are square-free}, and that {\em every \PPN{} except $2$ is pseudoperfect}. As with perfect numbers, it is unknown whether there are infinitely many PPNs or any odd ones.

\begin{notation*}
For an integer $r \ge 1$, we denote by $\K_r$ any \PPN{} with exactly $r$ (distinct) prime factors.
\end{notation*}

Remarkably, there exists precisely one $K_r$ for each positive integer $r\le 8$. This was conjectured by Ke and Sun \cite{Ke} and Cao, Liu, and Zhang \cite{Cao}, and then verified in  \cite{Butske} (see also Anne \cite{Anne}) using computational search techniques. Table~$1$ lists all known \PPN{s} and their prime factors.

\begin{center}
\begin{table}[ht] \label{TABLE: PPPs}
\begin{tabular}{|c|r|l|}
	\hline
	\textit{r} & \textit{$\K_r$} \hspace{7em} &  \hspace{2.8em}Prime Factorization  \\
	\hline
	$1$ & $2$ & $2$  \\
	$2$ & $6$ & $2\cdot3$  \\
	$3$ & $42$ & $2\cdot3\cdot7$  \\
	$4$ & $1806$ & $2\cdot3\cdot7\cdot43$  \\
	$5$ & $47058$ & $2\cdot3\cdot11\cdot23\cdot31$  \\
	$6$ & $2214502422$ & $2\cdot3\cdot11\cdot23\cdot31\cdot47059$  \\
	$7$ & $52495396602$ & $2\cdot3\cdot11\cdot17\cdot101\cdot149\cdot3109$  \\
	$8$ & $8490421583559688410706771261086$ & $2\cdot3\cdot11\cdot23\cdot31\cdot 47059 \cdot\, 2217342227\cdot1729101023519$  \\
	\hline
\end{tabular}
\vspace{0.5em}\caption{The \ppp{s} with $r\le8$ prime factors}
\end{table} 
\end{center}
\vspace{-3.25em}

Here are five related observations on Table~1 and Sylvester's sequence~\eqref{EQ: SS}.
\begin{enumerate}
	\item[(a).] $K_1 = 2$, $K_2= 2\cdot3=6$, $K_3= 6\cdot7=42$, and $K_4 = 42\cdot43=1806$, but $K_5 \ne 1806\cdot1807$.
	\item[(b).] $K_5 = 47058$ and $K_6 = 47058 \cdot 47059 = 2214502422$, but $K_7\ne 2214502422\cdot2214502423$.
	\item[(c).] $K_6 = 2214502422$ and $\K_8=2214502422\cdot2217342227\cdot1729101023519$.
	\item[(d).] $K_1,K_2,K_3,K_4=2,6,42,1806$ are each $1$ less than the terms $S_2,S_3,S_4,S_5=3, 7, 43, 1807$.
	\item[(e).] $K_r < S_{r+1}$, for $r=1,2,\dotsc,8$.
\end{enumerate}

These patterns can all be explained.

\begin{prop}  \label{PROP: primes to PPNs}
For any integer $K$, set $K':=K(K+1)$.
\begin{enumerate}
	\item Assume that $K+1$ is prime. Then $K$ is a \PPN{} if and only if $K'$ is also a \PPN{}.
	\item Assume that we can factor $K^2+1=(p-K)(q-K)$, for some primes $p>K$ and $q>K$. Then $K$ is a \PPN{} if and only if $K\cdot p\cdot q$ is also a \PPN{}.
	\item If $K+1=S_n$ is a term in Sylvester's sequence, then $K'+1=S_{n+1}$ is the next term in it.
	\item The inequality $K_r \le S_{r+1}-1$ holds for {\em any} \PPN{} with $r\ge1$ prime factors.
\end{enumerate}
\end{prop}
\begin{proof}
(i). This follows easily from Definition \ref{DFN: ppp} and the relation $\frac{1}{K'} = \frac{1}{K} - \frac{1}{K+1}$.\\
(ii). The proof is similar; for details, see Brenton and Hill's more general Proposition~12 in \cite{BH}, as well as \cite[Lemma~2]{Anne} and \cite[Lemma~4.1]{Butske}.\\
(iii). Sylvester's sequence satisfies $S_{n+1}=(S_n-1)S_n+1$. Setting $S_n=K+1$ gives (iii).\\
(iv). This follows directly from Curtiss's bound~\eqref{EQ: Curtiss bound}.
\end{proof}

Now, as $3,7,43,47059$ are prime, but $1807=13\cdot139$ and $2214502423=7^2\cdot45193927$ are composite, and as the numbers $2217342227$ and $1729101023519$ in the factorization
\begin{equation*}
	2214502422^2 + 1 = (2217342227 - 2214502422)(1729101023519 - 2214502422)
\end{equation*}
are prime, the observations (a), (b), (c), (d), and (e) are explained.

Analogs of (i) and (ii) for $K-1$ and $K^2-1$, involving \PPN{s} and Giuga numbers, are given in \cite[Theorem~8]{ST}.

\section{{\bf PPNs AND ARITHMETIC PROGRESSIONS.}}  \label{SEC: APS}

According to Table~$1$, the \PPN{s} having $r=2,3,4,5,6,7,8$ prime factors, \ie
\begin{align*}
	K_r =\ & 6, 42, 1806, 47058, 2214502422, 52495396602, 8490421583559688410706771261086,
\end{align*}
are all multiples of $2\cdot3=6$:
\begin{align*}
	\frac{K_r}{6} =\ &1, 7, 301, 7843, 369083737, 8749232767,  1415070263926614735117795210181.
\end{align*}

\begin{prop} \label{PROP: K == 6 mod 36}
Let $K$ be {\em any} \PPN{} divisible by $6$. Then $K \equiv 6\!\pmod{6^2}$.
\end{prop}
\begin{proof}
Denote by $\mu\ (\ge0)$ the number of prime factors of $K$ congruent to $-1$ modulo $6$. Since $6 \mid K$ and $K$ is square-free, $\frac{K}{6}\equiv (-1)^{\mu}\!\pmod{6}$. Now, reducing equation~\eqref{EQ: integer def} modulo $6$ gives
\begin{equation}
	1 + \frac{K}{2} + \frac{K}{3} + \sum_{3\mspace{2mu}<\mspace{3mu}p\mspace{3mu}\mid\mspace{1mu} K} \frac{K}{p}=K  \implies  1 + 3(-1)^{\mu}+ 2(-1)^{\mu} \equiv 0 \!\pmod{6} \label{EQ: K/6}
\end{equation}
and hence $\mu$ is even. This proves the proposition.
\end{proof}

In particular, for $r=2,3,4,5,6,7,8$ we find respectively that
\begin{align*}
	\frac{K_r-6}{6^2} &= 0, 1, 50, 1307, 61513956, 1458205461, 235845043987769122519632535030.
\end{align*}

Let us write $N\!\!\pmod{M} = R$ if the remainder upon division of $N$ by $M$ is $R$, so that both the congruence $N\equiv R\!\pmod{M}$ and the inequalities $0\le R<M$ hold. In light of Proposition~\ref{PROP: K == 6 mod 36} and the values $(K_2,K_3)=(6,42)$, one might predict that if we divide $K_2,\dotsc,K_8$ by some number $M$\hspace{-0.1em}, the remainders will form the arithmetic progression (\AP{} for short)
\begin{equation} \label{EQ:modM}
K_r\mspace{-15mu}\pmod{M}=6, 42, 78, 114,  150, 186,222, \text{ for }r=2,3,4,5,6,7,8,
\end{equation}
respectively. This requires $M$ to exceed $222$ and to divide each of the differences
\begin{align*}
	1806-78  =1728&=2^6\cdot3^3,  \\
	47058-114 = 46944 &= 2^5\cdot3^2\cdot163,  \\
	2214502422 - 150=2214502272&=2^7\cdot3^2\cdot89\cdot21599,  \\
	52495396602-186=52495396416&=2^6 \cdot3^2\cdot 47\cdot 1939103,  \\
	8490421583559688410706771261086 -222 &= 8490421583559688410706771260864  \\[-0.25em]
		&= 2^6\cdot3^2\cdot338293 \cdot43572628606668095873923.
\end{align*}
Since their greatest common divisor is $2^5\cdot3^2=288>222$, and no proper factor of $288$ exceeds $222$, the choice $M=288=6^2\cdot 8$ is both necessary and sufficient. This establishes a remarkable property of these \PPN{s}.

\begin{prop} \label{PROP: PPP AP}
Upon division of the \ppp{s} $K_2$, $K_3$, $K_4$, $K_5$, $K_6$, $K_7$, $K_8$ by $M=288$, the remainders form the $7$-term arithmetic progression \eqref{EQ:modM}, that is,
\begin{equation} \label{EQ:mod288}
	  K_r\mspace{-15mu}\pmod{6^2\cdot 8}= 6+6^2(r-2)\ \text{for}\ r=2,3,4,5,6,7,8.
\end{equation}
Moreover, no other modulus will do.
\end{prop}

Notice that the inequalities
\begin{equation*}
	6+6^2\cdot (9-2) = 258 < 288 < 294=6+6^2\cdot(10-2)
\end{equation*}
hold. Thus, the remainder pattern in \eqref{EQ:mod288} might persist for $r=9$ (assuming that a $K_9$ exists), but cannot for $r\ge10$. Throwing caution to the wind, we therefore make the following prediction.

\begin{cnj} \label{CNJ: PPP9}
There exists exactly one \ppp{} $\K_9$ with nine prime factors, and $\K_9\!\pmod{6^2\cdot 8}= 258$ holds. No further \PPN{s} exist.
\end{cnj}

Anyone thinking of settling Conjecture \ref{CNJ: PPP9} by computation
should be aware that Curtiss's upper bound for a ninth PPN is $K_9<S_{10},$ a $106$-digit number.

In case all or part of Conjecture \ref{CNJ: PPP9} fails, we also predict a strengthening of Proposition~\ref{PROP: K == 6 mod 36} for all \PPN{}s divisible by~$6$, including those with more than eight prime factors, if any.

\begin{cnj} \label{CNJ: PPPmod288}
For all $r\ge2$, if $6\mid K_r$, then $K_r \equiv 6+6^2(r-2)\!\pmod{6^2\cdot 8}$. Equivalently (by Proposition~\ref{PROP: K == 6 mod 36}), if $K_r > 2$, then $K_r$ is a multiple of $6$ and
\begin{equation*}
	\frac{K_r -6}{6^2} \equiv r-2\pmod{8}.
\end{equation*}
\end{cnj}

Note that the case $r=9$ here is weaker than Conjecture~\ref{CNJ: PPP9}. Note also that the quantity $r-2$ equals the number of prime factors of $K_r$ different from $2$ and $3$. Thus, each such factor conjecturally contributes $1$ to $(K_r - 6)/6^2$ modulo $8$ in some variant of the relation~\eqref{EQ: K/6}.

Although the modulus $6^2\cdot 8$ cannot be changed in Proposition~\ref{PROP: PPP AP}, other moduli provide interesting \AP{}s for subsets of the \PPN{s}. For example, we have \AP{s} of complementary subsequences $K_2$, $K_4$, $K_6$, $K_8\!\pmod{128}=6, 14, 22, 30$ and $K_3$, $K_5$, $K_7\!\pmod{128}=42, 82, 122$, so that
\begin{align} \label{EQ:mod128}
	K_r\mspace{-15mu}\pmod{2^7} =
	\begin{cases}
		\, 6+4(r-2)& \text{ for }r=2,4,6,8,  \\[0.15em]
		\, 42+20(r-3)& \text{ for }r=3,5,7.
	\end{cases}
\end{align}

Finally, we give a way to generate triples of \PPN{}s congruent modulo $6^3\cdot4=864$ to $3$-term \AP{s}.

\begin{prop} \label{PROP:triple}
Let $K$ be a \PPN{} such that $K+1$ and $K^2 + K+1$ are prime. Then the products $K':=K(K+1)$ and $K'':=K'(K'+1)$ are also \PPN{}s, and 
\begin{align} \label{EQ:mod432}
	K\equiv 0\hspace{-0.75em}\pmod{6} \implies K,K',K''\equiv K,K+6^2,K+6^2\cdot 2 \hspace{-0.75em}\pmod{6^3\cdot4},
\end{align}
respectively.
\end{prop}
\begin{proof}
Since $K+1$ and $K'+1=K^2 + K+1$ are prime, Proposition~\ref{PROP: primes to PPNs} part (i) implies that $K'$ and $K''$ are also \PPN{}s. As $6\mid K$, Proposition~\ref{PROP: K == 6 mod 36} gives $K=6+ 6^2n$, for some~$n$. Now, we can write
\begin{equation*}
	K'-K = K^2 =6^2+ 6^2\cdot4\cdot3n(3n+1)\equiv6^2\!\pmod{6^3\cdot4},
\end{equation*}
because $3n(3n+1)$ is even. In the same way we get $K''-K'\equiv 6^2\!\pmod{6^3\cdot4}$, and \eqref{EQ:mod432} follows.
\end{proof}

The only known example of Proposition~\ref{PROP:triple} is with $K=6$. The \ppp{}s $K,K',K''$ are then $6,42,1806$, whose remainders modulo $6^3\cdot4$ form the $3$-term arithmetic progression $6,42,78$. Compare to Proposition~\ref{PROP: PPP AP} for $r=2,3,4$.

It would be interesting to find explanations and extensions to all \PPN{}s, analogous to the statements and proofs of Propositions~\ref{PROP: primes to PPNs}, \ref{PROP: K == 6 mod 36}, and~\ref{PROP:triple}, for the \AP{}s of certain $K_r$ modulo $6^2\cdot 8$ and $2^7$ in \eqref{EQ:mod288} and \eqref{EQ:mod128}, respectively.

\section{{\bf THE ERD\H{O}S-MOSER CONJECTURE AND A CONDITIONAL RABBIT.}}  \label{SEC: EMC}
Erd\H{o}s and Moser (EM for short) studied equation~\eqref{EQ: EME} around $1953$ and made the following prediction.

\begin{cnj}[{\bf EM}] \label{CNJ: EMC}
The only solution to the EM equation~\eqref{EQ: EME} in positive integers is the trivial solution \mbox{$1^1 + 2^1 = 3^1$}.
\end{cnj}

Moser proved the following result toward Conjecture \ref{CNJ: EMC}.

\begin{thm}[{\bf Moser \cite{Moser}}] \label{THM: Moser}
If $(k,n)$ is a non-trivial solution of \eqref{EQ: EME}, then $k>10^{10^{6}}$.
\end{thm}

This bound was improved to $k>10^{1.485\times9321155}$ in \cite{Butske}, and to $k>10^{10^{9}}$ by Gallot, Moree, and Zudilin \cite{GMZ} (see also \cite[Chapter~8]{BPSZ}). On the other hand, it is not even known whether the number of solutions is finite. See the surveys~\cite[D7]{Guy} and~\cite{MoreeTophat}.

In \cite{sm} the authors approximated the EM equation by the \emph{EM congruence}
\begin{equation}
    1^n + 2^n + \dotsb +(k-1)^n+ k^n \equiv (k+1)^n \pmod{k}, \label{EQ: EMCong}
\end{equation}
as well as by the supercongruence modulo $k^2$, and proved the following connection with \PPN{}s.

\begin{prop} \label{LEM: which k and n}
The EM congruence~\eqref{EQ: EMCong} holds if and only if the inclusion
\begin{equation}
    \frac{1}{k} + \sum_{\subPk} \frac{1}{p} \in \mathbb{Z}  \label{EQ: frac mod 1}
\end{equation}
is true and $p\mid k$ implies $(p-1)\mid n$. In particular, every \ppp{} $K$ provides a solution $k := K$ to~\eqref{EQ: EMCong} with exponent $n:= \lcm\{p-1:p\mid K\}$. 
\end{prop}

Part of this is implicit in \cite{Moser}: Moser's work shows that \eqref{EQ: EME} implies \eqref{EQ: frac mod 1}; see \cite[p. 409]{Butske}.

In \cite{MoreeTophat} Moree wrote, ``In order to improve on [Theorem~\ref{THM: Moser}] by Moser's approach one needs to find additional rabbit(s) in the top hat. The interested reader is wished good luck in finding these elusive animals!'' Moree's top hat is a von Staudt-Clausen type theorem. Instead, we find a conditional rabbit in a hypothesis weaker than Conjecture~\ref{CNJ: PPP9}.

\begin{prop} \label{PROP:EM bound}
If there are no primary pseudoperfect numbers $K_r$ with $r \ge 33$, and if the \EME~\eqref{EQ: EME} has a non-trivial solution $(k,n)$, then $k>10^{3.99\times10^{20}}\!$.
\end{prop}

\begin{proof}
In \cite[Section 5.1]{GMZ} it is shown that if $(k,n)$ is a solution of \eqref{EQ: EME} with $n>1$, then the number of distinct prime factors of $k$ is at least $33$. Thus if no $K_r$ exists with $r\ge33$, then by Proposition~\ref{LEM: which k and n} the left-hand side of \eqref{EQ: frac mod 1}
cannot equal 1 and so, being a positive integer, must be $\ge2$. In the analysis of Moser's proof, this leads now to the inequality
\begin{equation}
	 \frac{1}{m-1} + \frac{2}{m+1} + \frac{2}{2m-1} + \frac{4}{2m+1} +  \sum_{p\mid M} \frac{1}{p}  \ge 4\frac{1}{6}  \label{EQ: Moser}
\end{equation}
(instead of $\ge3\frac{1}{6}$ as in \cite[equation (14)]{MoreeTophat} and \cite[equation (19)]{Moser}), where $m-1=k$ and $M=(m^2-1)(4m^2-1)/12.$ Now, $m-1=k>2^{33}>8\times10^9$ and so \eqref{EQ: Moser} implies
\begin{equation}
	\sum_{p\,\mid\, M} \frac{1}{p} >  4.166666. \label{EQ: 4.16}
\end{equation}
\emph{From \eqref{EQ: 4.16} it follows that $M>\prod_{p\le x} p$ if} $\sum_{p\le x} \frac{1}{p} < 4.166666$.
We show that the last inequality in turn holds if $x=x _0:= 3.6769\times10^{21}$. First, recall that the theorem of Mertens states that $\lim_{x\to\infty}(\sum_{p\le x} \frac{1}{p} - \log \log x) = B_1$, where $B_1 = 0.261497\dotso$ is Mertens's constant \cite[A077761]{Sloane}. Now, with $x=x_0$ compute Dusart's explicit form of Mertens's theorem \cite[Theorem 6.10]{Dusart}, namely,
\begin{align}
	\Bigg\lvert\sum_{p\, \le\, x} \frac{1}{p} - \log \log x - B_1 \Bigg\rvert \le  \frac{1}{10\log^2 x} + \frac{4}{15\log^3 x} \qquad (x \ge 10372).  \label{EQ: Dusart}
\end{align}

In \cite[Theorem 5.2]{Dusart} Dusart also proved that
\begin{align*}
	\sum_{p\, \le\, x} \log p > \left(1 - \frac{1}{\log^3 x}\right)x \qquad (x \ge 89 967 803).
\end{align*}
Hence
\begin{align*}
	\log M > \log\! \prod_{p\,\le\, x_0} \!p \,
=\sum_{p\,\le\, x_0} \!\log p > \left(1 - \frac{1}{\log^3 x_0}\right)x_0 > 3.6768\times10^{21}.
\end{align*}

\noindent Now, $3M < m^4 = (k+1)^4$, so $\log (k+1) > (\log3+\log M)/4 > 9.192\times10^{20}$. Therefore $k > e^{9.19\times10^{20}} > 10^{3.99\times10^{20}}\!$. This proves the proposition.
\end{proof}

\begin{remark*}
If we assume the Riemann Hypothesis, then we may replace~\eqref{EQ: Dusart} with Schoenfeld's conditional inequality \cite{Schoenfeld}
\begin{align*}
	\Bigg\lvert\sum_{p\,\le\, x} \frac{1}{p} - \log \log x - B_1 \Bigg\rvert \le \frac{3 \log x + 4}{8\pi \sqrt{x}} \qquad (x \ge 13.5)
\end{align*}
(see \cite[equation (7.1)]{BKS}), and infer that $\sum_{p\le x_1} \frac{1}{p} < 4.166666$ if $x _1:= 3.6847\times10^{21}$. Using $x_1$ in place of $x_0$ in the rest of the proof, we arrive at the slightly better, but doubly conditional bound $k>10^{4\times10^{20}}\!$.
\end{remark*}

\section*{{\bf Acknowledgments.}} The authors are very grateful to the referee for several  suggestions and references, especially those which led to Proposition~\ref{PROP:EM bound}. We thank Wadim Zudilin for an improvement in Conjecture~\ref{CNJ: PPPmod288}.


\bigskip

\noindent 209 West 97th Street, New York, NY, 10025\\
\noindent \url{jsondow@alumni.princeton.edu}\\

\noindent 55 Lessard Avenue, Toronto, Ontario, Canada~M6S 1X6\\
\noindent \url{kieren@alumni.rice.edu}


\begin{thebibliography}{99}

\bibitem{Anne} P. Anne, Egyptian fractions and the inheritance problem, \emph{College Math. J.} \textbf{29} (1998) 296--300.

\bibitem{ACF} R. Arce-Nazario, F. Castro, R. Figueroa, On the number of solutions of $\sum_{i=1}^{11} \frac{1}{x_i}=1$ in distinct odd natural numbers, \emph{J. Number Theory} \textbf{133} (2013) 2036--2046.

\bibitem{BKS} E.~Bach, D.~Klyve, J.~P.~Sorenson, Computing prime harmonic sums, \emph{Math. Comp.} \textbf{78} (2009) 2283--2305.

\bibitem{BBBG} D. Borwein, J. M. Borwein, P. B. Borwein, R. Girgensohn, Giuga's conjecture on primality, \emph{Amer. Math. Monthly} \textbf{103} (1996) 40--50.

\bibitem{BPSZ} J. Borwein, A. van der Poorten, J. Shallit, W. Zudilin, \emph{Neverending Fractions: An Introduction to Continued Fractions}. Cambridge Univ. Press, Cambridge, 2014.

\bibitem{BH} L. Brenton, R. Hill, On the Diophantine equation $1=\sum 1/n\sb i+1/\prod n\sb i$ and a class of homologically trivial complex surface singularities,  \emph{Pacific J. Math.} \textbf{133} (1988) 41--67.

\bibitem{BV} L. Brenton, A. Vasiliu, Znam's problem, \emph{Math. Mag.} \textbf{75} (2002) 3--11.

\bibitem{Butske} W.\ Butske, L.~M.\ Jaje, D.~R.\ Mayernik, On the equation $\sum_{p \mid N} \frac{1}{p} + \frac{1}{N} = 1$, pseudoperfect numbers, and perfectly weighted graphs, \emph{Math. Comp.} \textbf{69} (2000) 407--420.

\bibitem{Cao} Z. Cao, R. Liu, L. Zhang, On the equation $\sum_{j=1}^s (\frac1{x_j})+(\frac1{x_1 \cdots x_s})=1$ and Zn\'am's problem,
\emph{J.~Number Theory} \textbf{27} (1987) 206--211.

\bibitem{Curtiss} D. R. Curtiss, On Kellogg's Diophantine problem, \emph{Amer. Math. Monthly} \textbf{29} (1922) 380--387, \url{http://www.jstor.org/stable/2299023}.

\bibitem{Dusart} P. Dusart, Estimates of some functions over primes without R.H., preprint (2010), \url{http://arxiv.org/pdf/1002.0442v1}.

\bibitem{GMZ} Y. Gallot, P. Moree, W. Zudilin, The Erd\H{o}s-Moser equation $1^k + 2^k + \dotsb + (m-1)^k= m^k$ revisited using continued fractions, \emph{Math. Comp.} \textbf{80} (2011) 1221--1237.

\bibitem{Guy} R. K. Guy, \emph{Unsolved Problems in Number Theory}. Third ed. Springer, New York, 2004.

\bibitem{Ke} Z. Ke, Q. Sun, On the representation of $1$ by unit fractions, \emph{Sichuan Daxue Xuebao} \textbf{1}
(1964) 13--29.

\bibitem{Kellogg} O. D. Kellogg, On a Diophantine problem, \emph{Amer. Math. Monthly} {\textbf28} (1921) 300--303.

\bibitem{Landau} E. Landau, On the class number of binary quadratic forms of negative discriminant (in German), \emph{Math. Ann.} \textbf{56} (1903) 671--676.

\bibitem{Lenstra} H. Lenstra, Ode to the number 43 (in Dutch), \emph{Nieuw Arch. Wiskd.} \textbf{5} no. 10 (2009) 240--244.

\bibitem{MoreeTophat} P.~Moree, A top hat for Moser's four mathemagical rabbits, \emph{Amer.\ Math.\ Monthly} \textbf{118} (2011) 364--370, \url{http://arxiv.org/abs/1011.2956}.

\bibitem{Moser} L.\ Moser, On the Diophantine equation $1^n + 2^n + 3^n + \dotso + (m - 1)^n = m^n$, \emph{Scripta Math.} \textbf{19} (1953) 84--88.

\bibitem{pm} PlanetMath, Primary pseudoperfect number, \url{http://planetmath.org/primarypseudoperfectnumber}.

\bibitem{Schoenfeld} L. Schoenfeld, Sharper bounds for the Chebyshev functions $\theta(x)$ and $\psi(x)$. II, \emph{Math. Comp.} \textbf{30} (1976) 337--360.

\bibitem{Sloane} N. J. A. Sloane, The On-Line Encyclopedia of Integer Sequences, \url{https://oeis.org}.

\bibitem{sm} J.~Sondow, K.~MacMillan, Reducing the Erd\H{o}s-Moser equation \mbox{$1^n + 2^n + \dotsb + k^n= (k+1)^n$} modulo $k$ and $k^2$, \emph{Integers} \textbf{11} (2011) article A34, \url{http://www.integers-ejcnt.org/vol11.html}, expanded version \url{http://arxiv.org/abs/1011.2154}.

\bibitem{ST} J.~Sondow, E.~Tsukerman, The $p$-adic order of power sums, the Erd\H{o}s-Moser equation, and Bernoulli numbers, preprint (2014), \url{https://arxiv.org/abs/1401.0322}.

\bibitem{Sylvester} J. J. Sylvester, On a point in the theory of vulgar fractions, \emph{Amer. J. Math.} \textbf{3} (1880) 332--335, \url{http://www.jstor.org/stable/2369261}.

\bibitem{Weisstein} E. W. Weisstein, Primary pseudoperfect number\emdash From MathWorld, A~Wolfram Web Resource, \url{http://mathworld.wolfram.com/PrimaryPseudoperfectNumber.html}.

\bibitem{wiki} Wikipedia, Primary pseudoperfect number, {\it Wikipedia, the Free Encyclopedia}, \url{http://en.wikipedia.org/wiki/Primary_pseudoperfect_number}.

\end{thebibliography}
\end{document}